\documentclass{amsart}
\usepackage{amsmath, amsthm, amssymb,slashed}

\usepackage[usenames, dvipsnames]{color}
\usepackage[svgnames]{xcolor}
\usepackage[colorlinks,citecolor=RoyalBlue, urlcolor=RoyalBlue, linkcolor=RoyalBlue ]{hyperref} %BlueViolet  NavyBlue RoyalBlue MidnightBlue

\usepackage[normalem]{ulem}

\usepackage{booktabs}

\definecolor{mygray}{gray}{0.6}

\usepackage{upgreek}
\usepackage{bbm}

\newenvironment{myfont}[2][]{\csname#2\endcsname[#1]}{}

\usepackage{slashed}
\usepackage[makeroom]{cancel}
\usepackage[normalem]{ulem}
\usepackage{soul}
\newcommand{\stkout}[1]{\ifmmode\text{\sout{\ensuremath{#1}}}\else\sout{#1}\fi}

\usepackage{sseq}
\usepackage[all,cmtip]{xy}
\usepackage{tikz-cd}
\usepackage{tikz}
\usetikzlibrary{calc}
\usetikzlibrary{matrix}
\usetikzlibrary{decorations.markings}
\usetikzlibrary{tikzmark,decorations.pathreplacing,positioning}
\usepackage{amsfonts}

\newcommand{\bea}{\begin{eqnarray}}
\newcommand{\eea}{\end{eqnarray}}
\def\be{\begin{equation}}
\def\ee{\end{equation}}

\definecolor{red}{rgb}{1,0,0}
\definecolor{blue}{rgb}{0,0,1}
\definecolor{dblue}{rgb}{0,0,0.4}
\definecolor{green}{rgb}{0,1,0}
\definecolor{black}{rgb}{0,0,0}
\definecolor{white}{rgb}{1,1,1}

\definecolor{brn}{rgb}{.8,.4,.0}
\definecolor{redo}{rgb}{1,.5,.0}
\definecolor{ddgrn}{rgb}{0,0.4,0}
\definecolor{dgrn}{rgb}{0,0.55,0}
\definecolor{dbl}{rgb}{0,0,0.5}

\usepackage[bbgreekl]{mathbbol}
\usepackage{amscd}

\newcommand{\C}{\mathbb{C}}
\newcommand{\R}{\mathbb{R}}

\newcommand{\bpm}{\begin{pmatrix}}
\newcommand{\epm}{\end{pmatrix}}
\newcommand{\bmm}{\begin{matrix}}
\newcommand{\emm}{\end{matrix}}

\def\R{{\mathbb{R}}}
\def\C{{\mathbb{C}}}

\def \H{\mathbb{H}}

\newcommand {\emptycomment}[1]{}

\def\B{\mathrm{B}}

\usepackage{centernot}
\newcommand{\nn}{{\nonumber}}
\newcommand{\Sec}[1]{Sec.~\ref{#1}}

\usepackage{enumitem} %{enumerate}
\usepackage{mathtools,amssymb,varwidth}

\usepackage{datetime}

\newtheorem{theorem}{Theorem}[section]

\newtheorem{lemma}[theorem]{Lemma}

\newtheorem{definition}[theorem]{Definition}

\begin{document}

\title{The photography method: solving pentagon equation}

\author{Vassily Olegovich Manturov}

\address{Moscow Institute of Physics and Technology, Moscow 141700, Russia \\
Nosov Magnitogorsk State Technical University, Zhilyaev Laboratory of mechanics of gradient nanomaterials, 38 Lenin prospect, Magnitogorsk, 455000, Russian Federation \\
vomanturov@yandex.ru}

\author{Zheyan Wan}

\address{Yanqi Lake Beijing Institute of Mathematical Sciences and Applications, Beijing 101408, China \\
wanzheyan@bimsa.cn}

\maketitle

\begin{abstract}

In the present paper, we consider two applications of the pentagon equation. 
The first deals with actions of flips on edges of triangulations labelled by rational functions in some variables.
The second can be formulated as a system of linear equations with variables corresponding to triangles of a triangulation.
%This equation can be formulated pictorially as shown in Fig.~\ref{flip_5-gon}.
%%%%%%%%%%%% Refer to the figure with five flips and arrows.
The general method says that if there is some general {\em data} (say, edge lengths or areas)
associated with {\em states} (say, triangulations)
and a general {\em data transformation rule}
(say, how lengths or areas are changed under flips)
then after returning to the initial state we recover the initial data.

%The two formulas are given in \eqref{eq:Ptolemy} and \eqref{eq:matrix}.
%%%%%%%%%Refer to two formulas: x--> y= (ac+bd)/x
%%%%%%%%%%%% and Korepanov's matrix

%______________________________
%%%%%%This will appear if we do find the hexagon quickly

%%%%%%%%%%%% In the last section, we discuss possible 
%%%%%%applications of the above method to other equations.

%In the present paper, we show how the principle ``if data from two pictures allows one
%to restore the whole object'' allows one to solve pentagon and other equations.
%
%Here we show that concrete solutions, both old and new ones, to some equations 
%(pentagon, hexagon, etc., mostly having close ties with mathematical physics)
%can be found and some invariants can be obtained \cite{KMNK,ManturovNikonovMay2023}.
%
%We give an explicit algorithm that can be generally done in other situations and
%formulate further directions and unsolved problems. 

\end{abstract}



\section{Introduction}

%\section{Why are the groups $\Gamma$ ubiquitous?}\label{sect:ubiquitous_gamma}

Many structures in mathematics are of extreme importance not because they just
solve a problem or provide nice invariants (which may also be the case) but because they
uncover close ties between various branches of mathematics. This allows
one to construct a ``universal translator'' between different languages different
mathematicians speak.

 %%%%%%%%%%%%%%%%%%%%%%%%%%%%%%%%
%%%%%%%%%%%%%%%%%%%%%%%%%Here you refer to our FIRST paper about $\Gamma$.

\begin{figure}[h]
\centering\includegraphics[width = 0.9\textwidth]{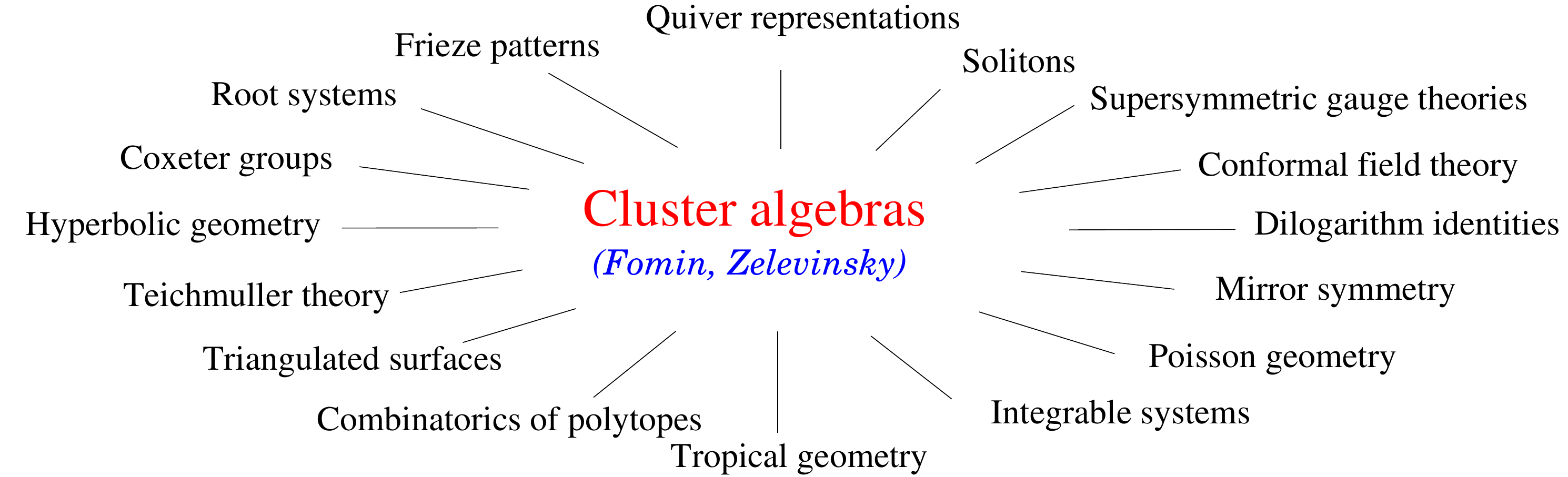}
\caption{Cluster algebras' connections and applications~\cite{Felikson}, see also \cite{FC,FZ1,FZ2,GKZ,KL,KR,Matveev,TV}.}\label{fig:cl_alg}
\end{figure}

Among such famous objects, we have 
%already mentioned 
cluster algebras
(see Fig.~\ref{fig:cl_alg}). We also mention Coxeter groups and Associahedra (Stasheff polytopes).
Among other branches of mathematics, the groups $\Gamma$ are related at
least to these three ones.

The groups $\Gamma_{n}^{k}$ were first introduced in
\cite{FMN} and intensively studied in
\cite{FKMN}.
In \cite{ManturovNikonovMay2023},
the groups $\Gamma_{n}^{k}$ 
motivated the authors to look for representations and actions of the braid group.
This leads to the {\em photography method}.
By the photography method (in a broad sense) we mean two very general methods for solving equations and
constructing invariants of objects.

For invariants, we consider objects (say, manifolds) represented by states (say, triangulations) modulo 
local moves (say, Pachner moves). 
With each state we associate some data (say, lengths, areas, angles, cocylces) and write down
some equations so that the set of solutions (maybe with some structure on it) before applying a move is in bijection with the set of solutions (with the same structure) after the move.
The equations are taken from some {\em natural data transmission rules} which are local,
so that the invariance evidently follows from geometrical reasons.

For equations, we take some {\em data} (say, edge-lengths) as variables, and we 
define {\em data transmission law} in such a way that the composition of {\em local}
moves corresponding to {\em flips between states} (say, Pachner moves)
leads to the initial data.
So, the {\em composition of data transmission operators} applied to {\em any data}
is identical. 
%Among all possible presentations and actions of the groups $\Gamma_{n}^{k}$ we mention
%just one method that the first named author calls the {\em photography method}
%(first introduced by the first named author in \cite{ManturovNikonovMay2023}):
%{\em We take several (say, two) photos of some object, and if we see that some of them are sufficient
%to recover the whole picture, then we carefully write down the way one picture can
%be transformed to the other.}
%{\em We have two concrete formulas \eqref{eq:Ptolemy} and \eqref{eq:matrix} relating data in one state with the data in the other state.}
%%%%%Refer to two formulas:
%%%%%%}
%The Ptolemy transformation and the generalisation given in the second section are
%evidences of that.
For further generalisations of the photography method, see \cite{KMNK}.

The Ptolemy transformation \eqref{eq:Ptolemy}
%%%%%%%%%%%refer to the formula
satisfies the pentagon equation,
i.e., the composition of five Ptolemy transformations shown in Fig.~\ref{flip_5-gon}
%%%%%%%refer to the figure 
is the identity.
Namely, in order to show that {\em Ptolemy yields pentagon}, we need not calculate
anything (see~\cite{KM,FKMN,FMN,FKMN1}):
we just say that if two quadrilaterals sharing three vertices are both inscribed then
the whole pentagon is inscribed, and the {\bf data} concerning any of its quadrilaterals
can be restored from the data of the two quadrilaterals.

What do we mean by {\bf data}? In geometry, those can be edge lengths,
angles, dihedral angles, areas, volumes, etc.
But probably other {\bf data} can come from interesting number theory
%(say, quadruples, quintuples etc) 
and hence yield new invariants
of (at least) braids and 3-manifolds.

We expect that many representations of the groups $\Gamma$ and other groups (known or
still to be discovered) can be explained by using this method. Probably, many statements and relations in the cluster algebra theory
also follow from this method.

The general idea of the photography method belongs to the first named author,
as well as the {\bf possible strategy what to do in further directions} (see the last section).
Most of the calculations are due to the second named author.

The paper is structured as follows. In \Sec{sec:Ptolemy}, we give a solution to the pentagon equation using the Ptolemy identity and photography method. In \Sec{sec:area}, we give another solution to the pentagon equation using the areas and photography method. In \Sec{sec:further}, we give further directions on photography method.

\subsection{Acknowledgements}

On the way towards the present state of the art, the first named author did a lot of research
jointly with I.M. Nikonov; the paper and the book coauthored 
with Nikonov are just a tiny bit of it. The present paper could not be written without
Nikonov's collaboration.

The authors are extremely grateful to L.H. Kauffman for sharing their joy in working on this 
project, and to I.G. Korepanov who indicated to the first author the idea which led to 
the understanding of how much can be done with this ``photography method''. 
The authors are also grateful to Joshua Abraham, Zichang Han, and Seongjeong Kim for helpful discussion and comments.

The study was supported by the grant of Russian Science Foundation (No. 22-19-20073 dated March 25, 2022 ``Comprehensive study of the possibility of using self-locking structures to increase the rigidity of materials and structures''). 

%\section{Two examples of the solution to the Ptolemy equation}
\section{Two faces of the pentagon equation and solutions to them}\label{sec:Ptolemy}
%The goal of the present section is to give the feeling of the ``photography method'' for the
%pentagon equation. 

The goal of the present section and the next section
is to give two forms of the pentagon equation
(as an action on lengths and as a matrix-form action on areas)
and give a solution to them.
%We say that a polygon on the hyperbolic plane is {\em inscribed} if there is a
%circle, horosphere, equidistant or line passing through all of its points.
%Obviously, each triangle is inscribed.
Let 
\bea\label{eq:Ptolemy}
f(a,b,c,d,x)=\frac{ac+bd}{x}.
\eea
This is the well known Ptolemy formula for the other diagonal of the circumscribed quadrilateral with edge-lengths $a,b,c,d$
and one diagonal of length $x$. 

%Now, we define $f_{2}(a,b,c,d,x)$ as follows \cite{Felikson}. 
Now, we follow the discussion in \cite{Felikson}. 
Consider the hyperbolic plane $\H^2=\{z\in \C|Im(z)>0\}$ with metric $ds^2=\frac{dx^2+dy^2}{y^2}$ and curvature $-1$. In this model, geodesics are half-lines and half-circles perpendicular to the real axis. The map $f(z)=\frac{z-i}{z+i}$ transforms the upper half-plane model to the Poincar\'e unit disk model.

Let $A,B\in\partial \H^2$ be two points at the boundary of the hyperbolic plane. The hyperbolic distance between them is infinite, but the infinity is concentrated around $\partial \H^2$ and can be dealt with by using {\em horocycles} as follows (in the upper half-plane model of $\H^2$, a horocycle centered at $\infty$ is a horizontal line, horocycles centered at other points are circles tangent to $\partial H^2$). Choose horocycles $h_A$ and $h_B$ centered at $A$ and $B$ (see Fig.~\ref{horo} for the pictures in the upper half-plane and unit disk models). Let $l_{AB}$ be the signed distance between the horocycles $h_A$ and $h_B$ ($l_{AB}$ is zero if the horocycles are tangent and negative if the horocycles intersect each other). Denote $\lambda_{AB}=e^{l_{AB}/2}$, the {\em lambda length} of $AB$.

Given an {\em ideal} quadrilateral $ABCD$ in $\H^2$ (a quadrilateral with all vertices at $\partial \H^2$) and a choice of horocycles around each of the vertices, one can prove that the lambda lengths for $ABCD$ satisfy Ptolemy identity \cite{Penner}:
\bea\label{eq:hyperbolic-Ptolemy}
\lambda_{AB}\cdot\lambda_{CD}+\lambda_{BC}\cdot\lambda_{DA}=\lambda_{AC}\cdot\lambda_{BD}.
\eea

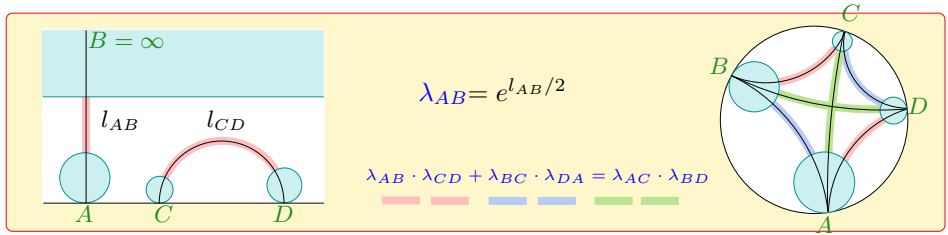
\begin{figure}[!h]
  \begin{center}
    \includegraphics[width=0.99\linewidth]{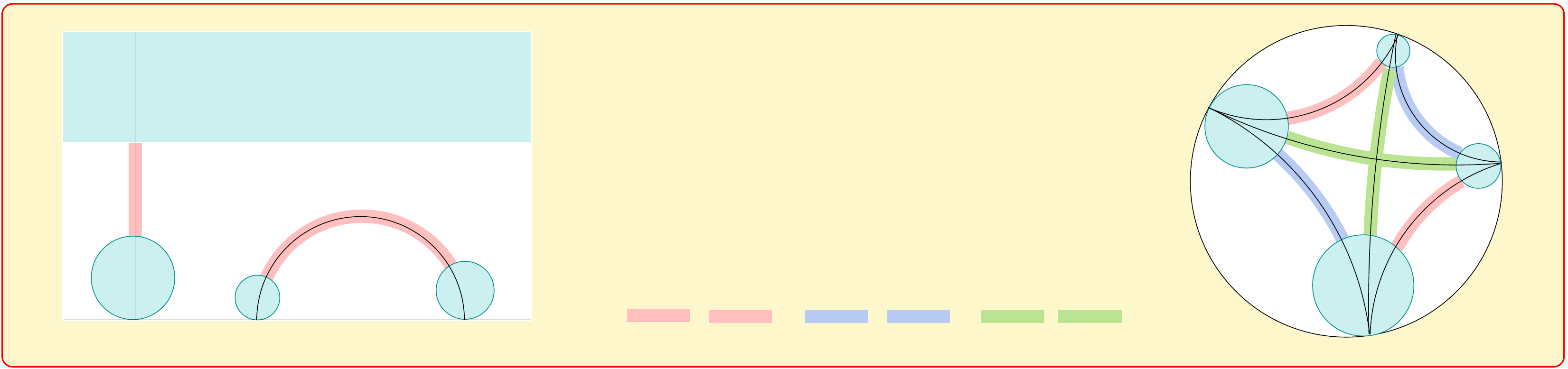}
    \put(-220,20){\tiny \color{blue}$\lambda_{AB}\cdot \lambda_{CD}+\lambda_{BC}\cdot \lambda_{DA}=\lambda_{AC}\cdot \lambda_{BD}$}
    \put(-200,50){{\color{blue} $\lambda_{AB}$}$=e^{l_{AB}/2}$}
    \put(-50,0){\small \color{ForestGreen} $A$}
    \put(-90,60){\small\color{ForestGreen} $B$}
    \put(-40,80){\small\color{ForestGreen} $C$}
    \put(-15,45){\small\color{ForestGreen} $D$}
    \put(-330,4){\small \color{ForestGreen} $A$}
    \put(-325,70){\small \color{ForestGreen} $B=\infty$}
    \put(-300,4){\small \color{ForestGreen} $C$}
    \put(-255,4){\small \color{ForestGreen} $D$}
    \put(-320,40){\small  $l_{AB}$}
    \put(-280,40){\small  $l_{CD}$}
    \caption{Lambda-lengths: removing infinity by horocycles} 
\label{horo}
\end{center}
\end{figure}

It is easy to show that given $a,b,c>0$, one can construct (in a unique way up to isometry) an ideal hyperbolic triangle with a choice of horocycles at its vertices such that $a=\lambda_{BC}$, $b=\lambda_{AC}$, and $c=\lambda_{AB}$. Also, triangles can be attached to each other along the edges with the same lambda lengths. Therefore, given a triangulated polygon with positive numbers assigned to its edges and diagonals in the triangulation, one can construct an ideal hyperbolic polygon with horocycles assigned to its vertices such that for every edge and diagonal the assigned number will coincide with the corresponding lambda length.
In particular, arbitrary positive numbers satisfying Ptolemy identity are realisable as the lambda lengths of the edges and diagonals of an ideal hyperbolic quadrilateral.

Consider two ideal hyperbolic triangles with lambda lengths of the edges $a,b,x$ and $c,d,x$
and glue them together along the edge $x$ to get an ideal hyperbolic quadrilateral with lambda lengths of the edges $a,b,c,d$ so that the edge of lambda length $b$ is adjacent to the edge of lambda length $c$.
%Set $f_{2}(a,b,c,d,x)$ to be the lambda length of the other diagonal of the quadrilateral.
By \eqref{eq:hyperbolic-Ptolemy}, $f(a,b,c,d,x)=\frac{ac+bd}{x}$ is the lambda length of the other diagonal of the quadrilateral.

Given a pentagon with vertices $A_{1},A_{2},A_{3},A_{4},A_{5}$.
Associate variables $a,b,c,d,e$ to the edges $A_{1}A_{2},\cdots, A_{5}A_{1}$.
Associate variables $x,y$ to the diagonals $A_{1}A_{3}$ and $A_{1}A_{4}$.

Start doing the following flips. First, we leave the diagonal $A_{1}A_{3}$ untouched
with label $x$; replace $A_{1}A_{4}$ with with $A_{3}A_{5}$ and put label
$z=f(x,c,d,e,y)$ on the new diagonal. 

We get two diagonals with labels on them; these labels are $x$ and some function on it.
Perform this operation five times (each time to a different quadrilateral) until we get back to the diagonals $A_{1}A_{3}$ and $A_{1}A_{4}$, see Fig.~\ref{flip_5-gon}.
At each step we get functions of $a,b,c,d,e,x,y$.

\begin{figure}[h]
\centering\includegraphics[width = 0.9\textwidth]{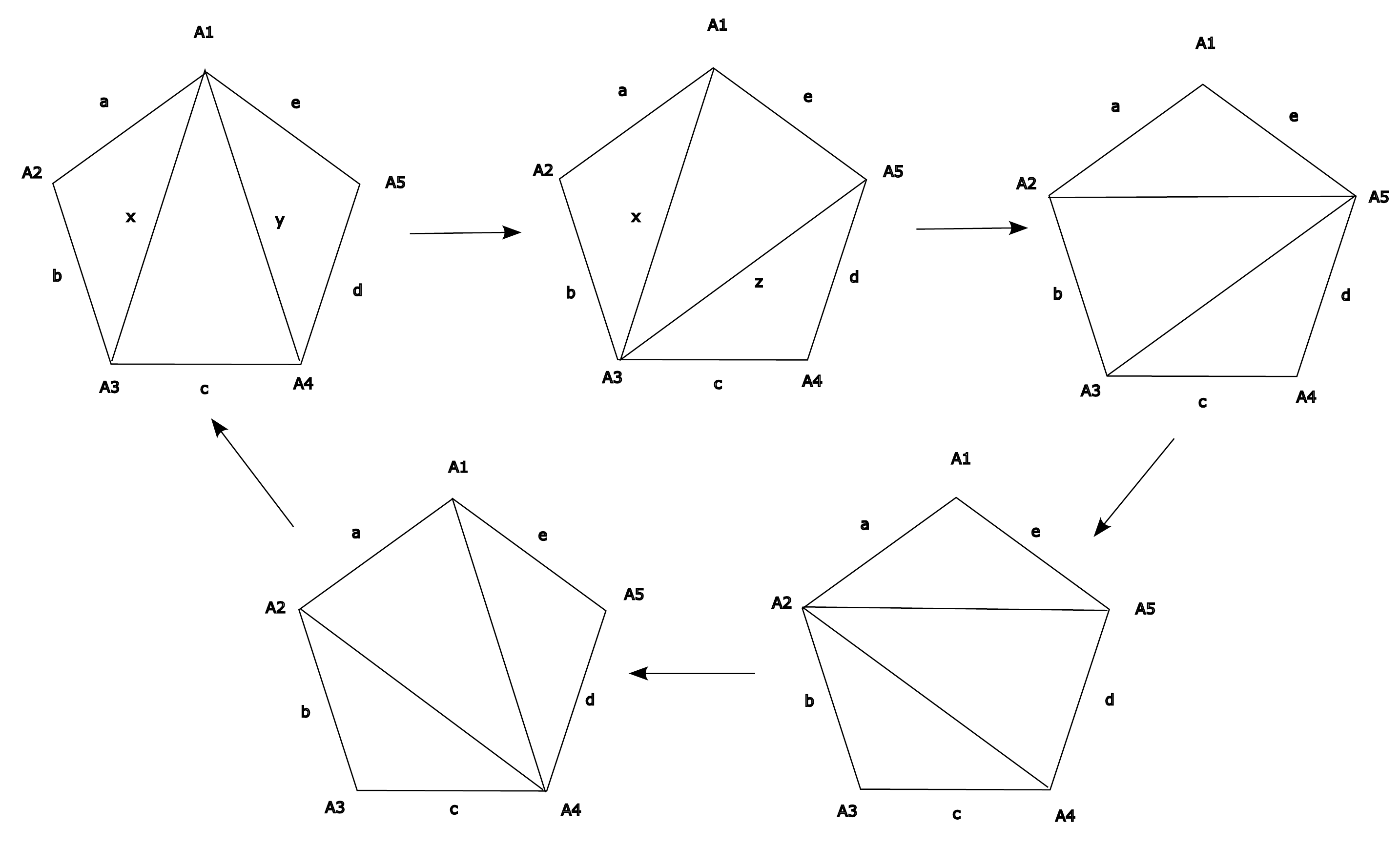}
\caption{Flip the diagonals of a pentagon five times.}\label{flip_5-gon}
\end{figure}

\begin{lemma}\label{lm1}
The new labels will be $x$ for $A_{1}A_{3}$ and $y$ for $A_{1}A_{4}$.
\end{lemma}

\begin{proof}
We just use the following obvious fact. If seven lengths are given arbitrarily, then one can construct an ideal hyperbolic pentagon with a choice of horocycles at its vertices such that the lambda lengths of the five edges and two diagonals are exactly the given seven lengths. This can be done by attaching three triangles to each other along the edges with the same lambda lengths.
%The rest is left to the reader. 
Then we apply the Ptolemy identity \eqref{eq:hyperbolic-Ptolemy} five times (each time to a different quadrilateral) and complete the proof.

%We just use the following obvious fact. If some four distinct points $B_{1},B_{2},B_{3},B_{4},B_{5}$
%on the hyperbolic unit disk are such that the $B_{1}B_{2}B_{3}B_{4}$ is an ideal quadrilateral 
%and $B_{1}B_{2}B_{3}B_{5}$ is an ideal quadrilateral. Then the whole pentagon $B_{1}B_{2}B_{3}B_{4}B_{5}$
%is an ideal pentagon.
%
%The rest is left to the reader. 

\end{proof}

By the pentagon equation we mean that the composition of five transformation
shown in Fig.~\ref{flip_5-gon} 
is the identical transformation.
This equation can be written in various forms depending on what these operators act on.

From this, we can get solutions to the pentagon equation and actions of the braid group 
by using the function $f$. The function $f$ gives an action 
of the pure braid group \cite{FKMN1} on labelled triangulations where labels of edges remain rational functions.
%This gives a nice action of the braid group on labelled triangulations.
%For $f_{2}$ the formula is more complicated (it involves the choice of horocycles). 
The beauty of the above observation is that
{\bf for the action of the braid group, we need not know what the formula really is, we only need to have Lemma \ref{lm1}.}

The action of the pure braid group on labelled triangulations is an instance of the photography method. The data are the lengths and the data transmission rule is the Ptolemy identity. We start from the data and data transmission rule. They are natural because they correspond to some geometry. Finally we end up with an action of the pure braid group.

We can do other things with other data, say, with angles of hyperbolic triangles
(the sum of angles is not $\pi$) or with triangles on the round sphere. For example, one can take angles of triangles and sets of axioms coming from Euclidean geometry (opposite angles of an inscribed quadrilateral sum to $\pi$) or Hyperbolic geometry (sums of opposite angles of an inscribed quadrilateral are equal) 
to write down the data transmission rules coming from flips.

%\section{The groups $\Gamma_n^4$}
%In this section, we recall the definition of the groups $\Gamma_n^4$ given in \cite{FMN,FKMN}.

%\begin{definition}\label{def-n-k}
%Let $4\le k\le n$. The group $\Gamma_n^k$ is the group with generators
%$$a_{P,Q}, \ P,Q\subset\{1,\dots,n\}, \ P\cap Q=\emptyset, \ |P\cup Q|=k, \ |P|,|Q|\ge2$$
%and the relations:
%\begin{enumerate}
%\item
%$a_{Q,P}=a_{P,Q}^{-1}$;
%\item
%\emph{far commutativity}: $a_{P,Q}a_{P',Q'}=a_{P',Q'}a_{P,Q}$ for each generators $a_{P,Q}, a_{P',Q'}$ such that
%$$|P\cap(P'\cup Q')|<|P|,\ |Q\cap(P'\cup Q')|<|Q|,$$
%$$|P'\cap(P\cup Q)|<|P'|,\ |Q'\cap(P\cup Q)|<|Q'|;$$
%\item
%\emph{$(k+1)$-gon relations}: for any standard Gale diagram (see \cite{FMN,FKMN} for the definition) $\bar{Y}$ of order $k+1$ and any subset $M=\{m_1,\dots,m_{k+1}\}\subset\{1,\dots,n\}$,
%$$\prod_{i=1}^{k+1}a_{M_R(\bar{Y},i), M_L(\bar{Y},i)}=1,$$
%where $M_R(\bar{Y},i)=\{m_j\}_{j\in R_{\bar{Y}}(i)}$, $M_L(\bar{Y},i)=\{m_j\}_{j\in L_{\bar{Y}}(i)}$.
%\end{enumerate}
%\end{definition}

%In particular, for $k=4$, Definition \ref{def-n-k} becomes Definitions \ref{def-n-4}.

\section{Towards actions of the braid group inspired by $\Gamma_n^4$}\label{sec:area}

In this section, we mention a way to get an action of the braid group inspired by the group $\Gamma_{n}^{4}$ which is related
to a solution of the {\em pentagon equation}, which was communicated to the first named author
by I.G. Korepanov. After several conversations, the authors realized that this solution
can be interpreted as a partial case of the ``photography method''.
First, we recall the definition of the group $\Gamma_n^4$ given in \cite{FMN,FKMN}.

\begin{definition}\label{def-n-4}
The group $\Gamma_n^4$ is the group generated by
$$\{d_{(ijkl)}\mid \{i,j,k,l\}\subset\{1,\dots,n\},|\{i,j,k,l\}|=4\}$$ subject to the following relations:
\begin{enumerate}
\item
$d_{(ijkl)}^2=1$ for $\{i,j,k,l\}\subset\{1,\dots,n\}$,
\item
$d_{(ijkl)}d_{(stuv)}=d_{(stuv)}d_{(ijkl)}$ for $|\{i,j,k,l\}\cap\{s,t,u,v\}|<3$,
\item
$d_{(ijkl)}d_{(ijlm)}d_{(jklm)}d_{(ijkm)}d_{(iklm)}=1$ for distinct $i,j,k,l,m$,
\item
$d_{(ijkl)}=d_{(kjil)}=d_{(ilkj)}=d_{(klij)}=d_{(jkli)}=d_{(jilk)}=d_{(lkji)}=d_{(lijk)}$ for distinct $i,j,k,l$.
\end{enumerate}
\end{definition}
Here $d_{(ijkl)}$ corresponds to the flip of diagonals in the quadrilateral $ijkl$. The third relation is the pentagon relation.

%The generators of $\Gamma_n^4$ correspond to flips of diagonals in a quadrilateral.

We will associate a $2\times 2$ matrix to each generator of $\Gamma_n^4$. Below we construct the matrix for $d_{(1234)}$, the matrices for other $d_{(ijkl)}$ are constructed similarly.
The two diagonals divide a quadrilateral $1234$ into 4 parts, the areas of the 4 parts are denoted by $A,B,C,D$, see Fig.~\ref{flip}.
\begin{figure}[h]
\centering\includegraphics[width = 0.9\textwidth]{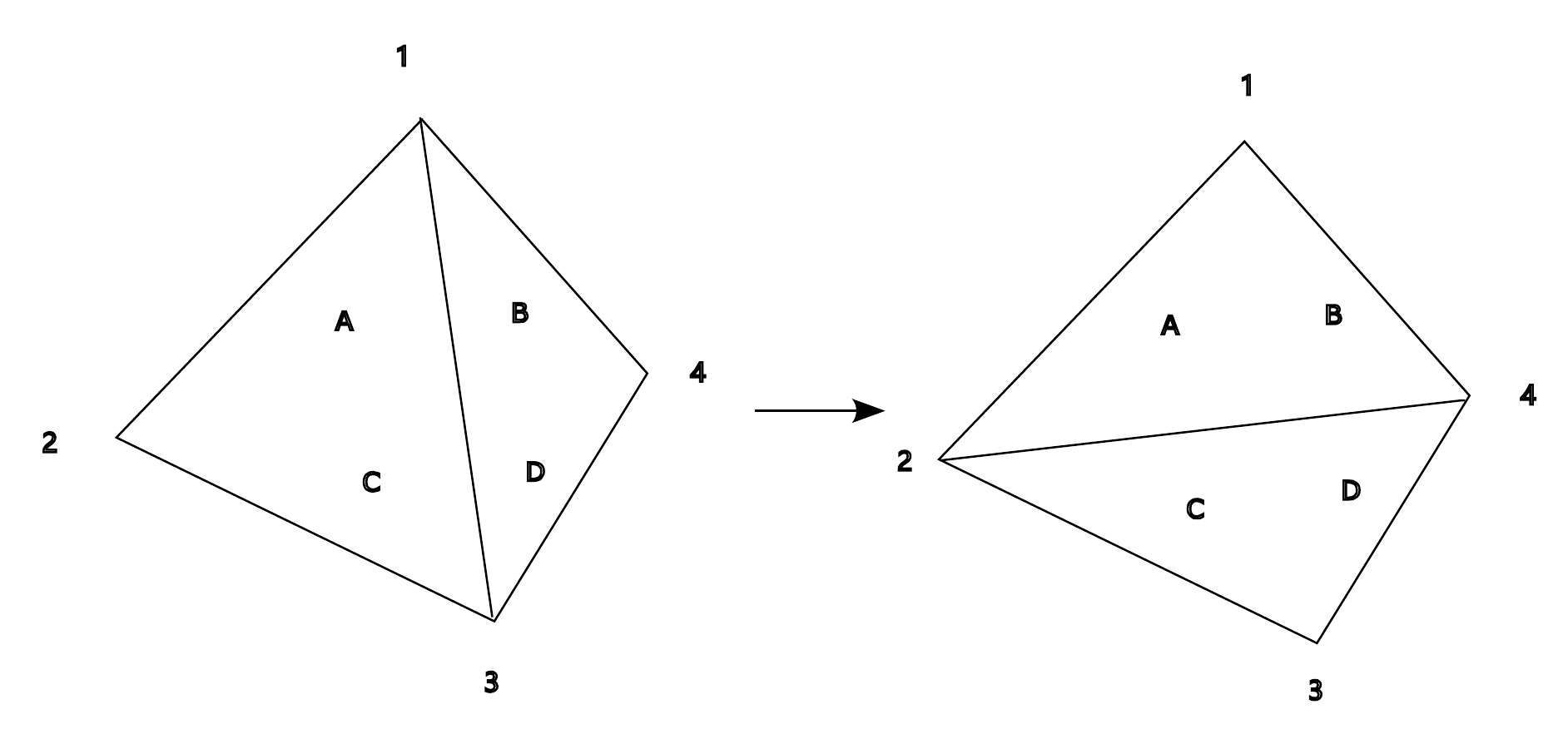}
\caption{Flip the diagonals of a quadrilateral.}\label{flip}
\end{figure}

Then the generator $d_{(1234)}$ flips the diagonal 13 to 24 of the quadrilateral 1234.
It maps $\left(\begin{array}{cc}A+C\\B+D\end{array}\right)$ to $\left(\begin{array}{cc}A+B\\C+D\end{array}\right)$. We try to associate a matrix $\left(\begin{array}{cc}a&b\\c&d\end{array}\right)$ to $d_{(1234)}$ such that
$$\left(\begin{array}{cc}a&b\\c&d\end{array}\right)\left(\begin{array}{cc}A+C\\B+D\end{array}\right)=\left(\begin{array}{cc}A+B\\C+D\end{array}\right).$$
Then 
$$
(1,1)\left(\begin{array}{cc}A+C\\B+D\end{array}\right)=(1,1)\left(\begin{array}{cc}A+B\\C+D\end{array}\right)
$$
for all $A,B,C,D$
implies that
$(a+c,b+d)=(1,1)$.

The inverse of the generator $d_{(1234)}$ flips the diagonal 24 to 13 of the quadrilateral 1234.
The matrix associated to $d_{(1234)}^{-1}$ is
$$\left(\begin{array}{cc}a&b\\c&d\end{array}\right)^{-1}=\left(\begin{array}{cc}a&b\\1-a&1-b\end{array}\right)^{-1}=\left(\begin{array}{cc}\frac{1-b}{a-b}&\frac{-b}{a-b}\\\frac{a-1}{a-b}&\frac{a}{a-b}\end{array}\right)=:\left(\begin{array}{cc}a'&b'\\c'&d'\end{array}\right).$$
We find that $(a'+c',b'+d')=(1,1)$.
Note that the matrices associated to $d_{(1234)}$ and $d_{(1234)}^{-1}$ are not equal in general.

Let $\zeta_i$ ($i=1,2,3,4$) be variables associated to the vertices of the quadrilateral, then we can construct the matrix of $d_{(1234)}$ explicitly as
\bea\label{eq:matrix}
\left(\begin{array}{cc}\frac{\zeta_3-\zeta_2}{\zeta_4-\zeta_2}&\frac{\zeta_1-\zeta_2}{\zeta_4-\zeta_2}\\\frac{\zeta_3-\zeta_4}{\zeta_2-\zeta_4}&\frac{\zeta_1-\zeta_4}{\zeta_2-\zeta_4}\end{array}\right).
\eea
Note that the sum of the elements in each column of this matrix is 1.

We verify that the matrices associated to the flips of diagonals satisfy the pentagon relation (see Fig.~\ref{flip_5-gon}).
Namely,
\bea
a_{1345}a_{1235}a_{2345}^{-1}a_{1245}^{-1}a_{1234}^{-1}=I_3
\eea
where 
$a_{ijkl}$ is the $3\times 3$ matrix associated to the flip of diagonals $d_{(ijkl)}$ in the quadrilateral $ijkl$ which fixes a triangle in the pentagon $12345$.
Note that each triangulation of the pentagon consists of three triangles, each flip of the diagonals changes two of the three triangles and fixes the other one. The $3\times 3$ matrix $a_{ijkl}$ is obtained from the $2\times2$ matrix constructed for $d_{(ijkl)}$ by adding a diagonal entry 1 corresponding to the fixed triangle.
By direct computation, we verify that
\bea
&&\left(\begin{array}{ccc}1&&\\&\frac{\zeta_4-\zeta_3}{\zeta_5-\zeta_3}&\frac{\zeta_1-\zeta_3}{\zeta_5-\zeta_3}\\&\frac{\zeta_4-\zeta_5}{\zeta_3-\zeta_5}&\frac{\zeta_1-\zeta_5}{\zeta_3-\zeta_5}\end{array}\right)\left(\begin{array}{ccc}\frac{\zeta_3-\zeta_2}{\zeta_5-\zeta_2}&\frac{\zeta_1-\zeta_2}{\zeta_5-\zeta_2}&\\\frac{\zeta_3-\zeta_5}{\zeta_2-\zeta_5}&\frac{\zeta_1-\zeta_5}{\zeta_2-\zeta_5}&\\&&1\end{array}\right)\left(\begin{array}{ccc}1&&\\&\frac{\zeta_5-\zeta_2}{\zeta_4-\zeta_2}&\frac{\zeta_3-\zeta_2}{\zeta_4-\zeta_2}\\&\frac{\zeta_4-\zeta_5}{\zeta_4-\zeta_2}&\frac{\zeta_4-\zeta_3}{\zeta_4-\zeta_2}\end{array}\right)\cdot\nn\\
&&\cdot\left(\begin{array}{ccc}\frac{\zeta_5-\zeta_1}{\zeta_4-\zeta_1}&&\frac{\zeta_2-\zeta_1}{\zeta_4-\zeta_1}\\&1&\\\frac{\zeta_4-\zeta_5}{\zeta_4-\zeta_1}&&\frac{\zeta_4-\zeta_2}{\zeta_4-\zeta_1}\end{array}\right)\left(\begin{array}{ccc}\frac{\zeta_4-\zeta_1}{\zeta_3-\zeta_1}&\frac{\zeta_2-\zeta_1}{\zeta_3-\zeta_1}&\\\frac{\zeta_3-\zeta_4}{\zeta_3-\zeta_1}&\frac{\zeta_3-\zeta_2}{\zeta_3-\zeta_1}&\\&&1\end{array}\right)=I_3.
\eea
Note that all generators of $\Gamma_n^4$ are involutions while these matrices are not.
So, formally, this is not a representation of $\Gamma_n^4$, though it does give rise to 
an action of the pure braid group on labelled triangulations (see \cite{FKMN}).

The Korepanov's solution is another instance of the photography method. The data are areas and the data transmission rule is \eqref{eq:matrix}. 
We start from the data and data transmission rule. They are natural because they correspond to some geometry. Finally we end up with an action of the pure braid group.

\section{Towards an algorithm: How it works and what to do further}\label{sec:further}

In this section we give a general strategy for how to solve various equations (like the pentagon) 
and give representations
of various groups (like $\Gamma$) by using {\em very mild data}.
%(like the axiom
%``for each three distinct points in Euclidean space there exists a unique circle or line
%passing through them'').
We consider {\em states} of an object in question (e.g., triangulations of a manifold).
We consider {\em moves between states} (say, Pachner moves, flips).
We consider some {\em data} together with a {\em data transformation law} which
should be {\em natural} and {\em sufficient}.
 
For data we can take, say, edge lengths (or $\lambda$-lengths), by a natural data transformation
law we consider, say, the Ptolemy transformation. It is {\em natural} because it originates from
some geometry, and the data is {\em sufficient} because the edge-lengths ($\lambda$-lenghs)
of {\em any} triangulation (state) uniquely define the data of any other state. See \Sec{sec:Ptolemy}.
%%%%%%%%%%%%%%%%%%%%%%%%%%Zheyan, here you refer to the section where it is done.

This leads to a formula (in our case, the Ptolemy formula) which can be treated as a
function (the way of rewriting lengths).
Geometric naturality guarantees that the function will satisfy certain equation
(in our case, Ptolemy yields pentagon).
 Hence, Lemma \ref{lm1}
%%%%%%Zheyan, refer to the best one
can be checked by a brute force algorithm \cite{FKMN}.
%%%%%%%%%%%%%%%%%%%%%%%Here you can refer to the book "Invariants and pictures".
 After such a purely algebraic proof, one can completely forget about what the initial objects
come from (lengths are natural numbers satisfying some mild conditions), and take abstract variables.

We also describe some possible caveats which may appear in constructing these solutions,
representations, and invariants, and give possible ways to overcome it. 
Also, we indicate possible applications of this method in various branches of mathematics.\\

1. Space with data.
Given a {\em homogeneous} space $S$ with {\bf data} which is invariant under the group of isometries.
{\bf Examples: Euclidean space, Lobachevsky space, Spheres, other homogeneous spaces.}\\

2. Nice family: for an integer $k$, on $S$ we have a family of 
``nice $k$-objects'' so that through each $k$  points there exists a unique nice $k$-object.
{\bf Examples:} $k=3$, Euclidean space, $2$-dimensional sphere, hyperbolic space.
%{\bf Refer to $G_{n}^{k}$.}
{\bf Caveats:} For $\R^{2}$ we need not only circles, but also lines, for the hyperbolic plane
we also need horocycles and equidistants. Moreover (see further) sometimes 
we have to deal with coincident points; in this case we have to justify our methods.\\

3. A formula for data (existence). By putting together two $k$-tuples of points (
two vertices of a $(k-1)$-simplex, we get a bypiramid where the whole data can be
restored from the data for the two $k$-tuples.
{\bf Examples:} Two inscribed triangles give rise to an inscribed quadrilateral,
where the second diagonal is expressed in terms of the first.\\

4. With or without family. Sometimes, the photography method works
without any ``nice family''.\\
%{\bf Example.} Areas of triangles and quadrilaterals give rise to hexagon equations.

5. From formula to function. Functions are sometimes algebraic; variables can be abstract.
The pentagon equation deals with sides of a pentagon on the Euclidean space, hence, with non-negative real numbers.
However, the implication ``Ptolemy yields Pentagon'' for
the function $x\to \frac{ac+bd}{x}$ can be proved by purely algebraic methods.

Hence, we can use it {\bf not for numbers}, but rather, for {\bf variables} or
other {\bf abstract objects.} This led to the action of pure braid groups 
\cite{FKMN,FMN} on triangulations with edges marked by Laurent polynomials.

Similarly, the beautiful formula from section
%%%%%%%%%%
 due to I.G. Korepanov \cite{K19} 
originates from {\bf areas of quadrilaterals}
\footnote{This {\bf solution} was communicated to the first named author by
I.G. Korepanov, but the {\bf method of obtaining it} was recognized by the
first author by himself.
}, however for $\zeta_{i}$ one can take variables $\zeta$ from some field. Similar problems were studied in \cite{KS99,K04,KS13,K14}.

{\bf In two words:} one takes a fact known from geometry (maybe school geometry
like the Ptolemy for an inscribed quadrilateral, 
obtains some {\bf corollaries} from it which can be proved  {\bf elementarily
from the geometrical point of view}, and then since the equations hold,
they should have an algebraic proof!\\

6. {\bf Summary of caveats and ways to overcome them.}
One of the main difficulties the authors of
\cite{FKMN} met in knot theory was an attempt to pass from
braid invariants (there are lot of them coming from $\Gamma_{n}^{4}$)
to knot invariants. The problem was division by zero: 
the solution to the equation $xy=ac+bd$ in $y$ is not unique once
$ac+bd=0$.

However, one can apply either: 
a brute force methods
(as I.M. Nikonov did in \cite{ManturovNikonovMay2023}, see also \cite{M22,MN}) when
one deals with some rings, stabilisations, equivalence types
or
by using some methods like blowing up in algebraic geometry
(there is more than one line passing through two coinciding
points, but we can make points distinct by some blow-up).
This is not realised yet. \\

7. {\bf Further directions.}
In Fig.~\ref{fig:cl_alg} one can see the ``Cluster algebra flower''. Lots of leaves
of it are already related to the groups $\Gamma$.
In \cite{ManturovNikonovMay2023} one gets
braids and manifolds. 

Among further directions which may appear in further papers, we mention just two:
In topology: passing from braids to knots, passing from 3-manifolds to 4-manifolds etc.
In group theory: deeper understanding the structure of the group $\Gamma$ themselves.

In this paper we work with the pentagon relation, hence, we do something with
3-manifolds. 
The reader can choose a subject according to his/her preferences and try to find
the photography method there!


\begin{thebibliography}{99}



\bibitem{Felikson} A. Felikson, Ptolemy Relation and Friends, arXiv:2302.06379.

\bibitem{FC} V.V. Fock, L.O. Chekhov, A quantum Teichmüller space, \emph{Theor Math Phys} {\bf 120}:3 (1999) 1245--1259.

\bibitem{FZ1} S. Fomin and A. Zelevinsky, Cluster algebras I: Foundations, \emph{J. Amer. Math. Soc.} {\bf 15} (2002) 497--529.
\bibitem{FZ2} S. Fomin and A. Zelevinsky, Cluster algebras II: Finite type classification, \emph{Invent. Math.} {\bf 154}:1 (2003) 63--121.

\bibitem{GKZ} I.M. Gelfand, M.M. Kapranov, A.V. Zelevinsky, \emph{Discriminants, resultants, and multidimensional determinants}, Springer, 2008.



\bibitem{KL} L.H. Kauffman, S. Lins, \emph{Temperley--Lieb Recoupling Theory and Invariants of 3-Manifolds}, Princeton University Press, 1994.




\bibitem{KM} S. Kim, V.O. Manturov, Artin's braids, braids for three space, and groups $\Gamma_n^4$ and $G^k_n$,  {\it J. Knot Theory Ramifications} {\bf 28}:10 (2019) 1950063.


\bibitem{KMNK}
L. Kauffman, V. O. Manturov, I. M. Nikonov, S. Kim, {\em Photography principle, data transmission, and invariants of manifolds,} arXiv:2307.03437 


\bibitem{KR} A.N. Kirillov, N.Y. Reshetikhin, Representations of the algebra $U_q(sl_2)$, $q$-orthogonal polynomials and invariants of links, in \emph{Infinite dimensional Lie algebras and groups}, \emph{Adv. Ser. In Math. Phys.} {\bf 7}, World Scientific, Singapore (1988), 285--339.

\bibitem{KS99} I. G. Korepanov, S. Saito, Finite-dimensional analogues of the string $s\leftrightarrow t$ duality and the pentagon equation, \emph{Theoret. and Math. Phys.} {\bf 120}:1 (1999), 862--869.

\bibitem{K04} I.G. Korepanov, SL(2)-Solution of the Pentagon Equation and Invariants of Three-Dimensional Manifolds. \emph{Theoret. and Math. Phys.} {\bf 138}, 18–27 (2004). 


\bibitem{KS13} I.G. Korepanov, N.M. Sadykov, Pentagon Relations in Direct Sums and Grassmann Algebras, \emph{SIGMA} {\bf 9} (2013) 030.


\bibitem{K14} I.G. Korepanov, Two-cocycles give a full nonlinear parameterization of the simplest 3–3 relation, \emph{Lett Math Phys} {\bf 104} (2014) 1235--1261.

\bibitem{K19} I.G. Korepanov, Private communication to V.O. Manturov, 2019.

\bibitem{FKMN} V.O. Manturov, D. Fedoseev, S. Kim, I. Nikonov,  {\em Invariants And Pictures: Low-dimensional Topology And Combinatorial Group Theory, Series On Knots And Everything} {\bf 66}, World Scientific, 2020.

\bibitem{FMN} V.O. Manturov, D. Fedoseev, I. Nikonov, Manifolds of triangulations, braid groups of manifolds, and the groups $\Gamma_n^k$  In:\emph{Garanzha, V.A., Kamenski, L., Si, H. (eds) Numerical Geometry, Grid Generation and Scientific Computing. Lecture Notes in Computational Science and Engineering, vol 143.} Springer, Cham., 2021, 13--36.

\bibitem{FKMN1} V.O. Manturov, D. Fedoseev, S. Kim, I. Nikonov, On groups $G_n^k$ and $\Gamma_n^k$: A study of manifolds, dynamics, and invariants, \emph{Bull. Math. Sci.} {\bf 11}:2 (2021) 2150004.

\bibitem{M22} V.O. Manturov, Realisability of $G^3_n$, realisability projection, and kernel of the $G^3_n$-braid presentation, arXiv:2210.13338.

\bibitem{MN} V.O. Manturov, I. Nikonov, On an invariant of pure braids, arXiv:2303.04423.

\bibitem{ManturovNikonovMay2023} V.O. Manturov, I. Nikonov, The groups $\Gamma_n^4$, braids, and 3-manifolds, arXiv: 2305.06316.

\bibitem{Matveev} S. Matveev, Algorithmic topology and classification of 3-manifolds, Springer-Verlag Berlin Heidelberg, 2007.

\bibitem{Penner} R. Penner, The decorated Teichm\"uller space of punctured surfaces, Comm. Math. Phys. 113, n.2 (1987), 299--339.

\bibitem{TV} V. G. Turaev, O. Ya. Viro, State sum invariants of 3-manifolds and quantum $6j$-symbols, \emph{Topology} {\bf 31}:4 (1992), 865--902.

\end{thebibliography}
\end{document}